\documentclass[final,11pt,3p]{elsarticle}

\usepackage{graphicx}
\usepackage{amssymb}
\usepackage{amsthm}
\usepackage{amsmath}
\usepackage{algorithm}
\usepackage{algorithmic}
\usepackage{tikz}
\usepackage{lineno}
\usetikzlibrary{shapes}

\usepackage{subcaption}
\usepackage{caption}
\usepackage{float}
\usepackage{geometry}
\usepackage{mathrsfs} 
\usepackage{diagbox}
\biboptions{sort&compress}

\newtheorem{theorem}{Theorem}

\theoremstyle{definition}

\usepackage{xcolor}

\begin{document}

\begin{tikzpicture}[remember picture,overlay]
\node[anchor=north east,inner sep=20pt] at (current page.north east)
{\includegraphics[scale=0.2]{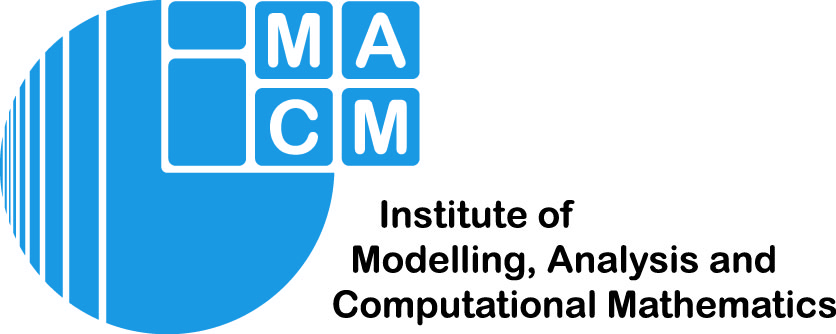}};
\end{tikzpicture}
	
\begin{frontmatter}

\title{Dynamical Analysis of a Cocaine-Heroin Epidemiological Model\\
with Spatial Distributions}

\author[AMNEA]{Achraf Zinihi}
\ead{a.zinihi@edu.umi.ac.ma} 

\author[AMNEA]{Moulay Rchid Sidi Ammi}
\ead{rachidsidiammi@yahoo.fr}

\author[BUW]{Matthias Ehrhardt\corref{Corr}}
\cortext[Corr]{Corresponding author}
\ead{ehrhardt@uni-wuppertal.de}

\author[CS]{Ahmed Bachir}
\ead{abishr@kku.edu.sa}

\address[AMNEA]{Department of Mathematics, MAIS Laboratory, AMNEA Group, Faculty of Sciences and Technics,\\
Moulay Ismail University of Meknes, Errachidia 52000, Morocco}

\address[BUW]{University of Wuppertal, Chair of Applied and Computational Mathematics,\\
Gaußstrasse 20, 42119 Wuppertal, Germany}

\address[CS]{Department of Mathematics, College of Science, King Khalid University,\\
Abha, Saudi Arabia.}


\begin{abstract}
This article conducts an in-depth investigation of a new spatio-temporal model for the cocaine-heroin epidemiological model with vital dynamics, incorporating the Laplacian operator. 
The study rigorously establishes the existence, uniqueness, 
non-negativity, and boundedness of solutions for the proposed model. 
In addition, the local stability of both a drug-free equilibrium and a drug-addiction equilibrium are analyzed by studying the corresponding characteristic equations. 
The research provides conclusive evidence that when the basic reproductive number $\mathcal{R}_0$ exceeds 1, the drug-addiction equilibrium is globally asymptotically stable. 
Conversely, using comparative arguments, it is shown that if $\mathcal{R}_0$ is less than 1, the drug-free equilibrium is globally asymptotically stable. 
Furthermore, the article includes a series of numerical simulations to visually convey and support the analytical results.
\end{abstract}




\end{frontmatter}

\section{Introduction}\label{S1}

In recent years, reports, e.g.\ from UN \cite{united2023world} or EU \cite{european2023european} have consistently highlighted the continued and persistent nature of cocaine and heroin use on a global scale.
Cocaine and heroin are both potent psychoactive substances with complex histories, each originating from different sources and cultures \cite{Booth1996, Gootenberg1999, OToole2015, Streatfeild2003}. 
Cocaine, derived from the leaves of the coca plant, has a rich history in South America, particularly among indigenous populations in regions such as the Andes. 
The use of coca leaves for medicinal and stimulant purposes dates back thousands of years among these communities. 
The indigenous people of the Andes chewed coca leaves to relieve fatigue and increase stamina, and the plant had cultural and religious significance. 
The earliest recorded use of coca leaves can be traced back to the ancient Inca civilizations \cite{Gootenberg1999, Streatfeild2003}.

Heroin, on the other hand, is a semi-synthetic opioid derived from Morphine, which in turn is derived from the opium poppy. 
The history of opium, and by extension heroin, is deeply rooted in ancient civilizations. 
The opium poppy was cultivated and used for its medicinal properties in ancient Mesopotamia, and the knowledge spread to various cultures, including ancient Greece and Rome. However, the synthesis of heroin itself is a more recent development, dating back to the late 19th century. 
Heroin was first synthesized by C.R. Alder Wright in 1874 and later re-synthesized by Heinrich Dreser at Bayer Laboratories in 1897. 
Its early use was as a potential substitute for morphine, which was widely used but also associated with addiction \cite{Booth1996, OToole2015}.

In recent decades, mathematical modeling has proven to be a valuable tool for studying infectious diseases that pose significant health threats to humanity. However, the societal impact of illicit drug abuse, including substances such as cocaine and heroin, demands increased attention from public health agencies. 
While mathematical models traditionally used for infectious diseases have been successfully applied to the study of drug-addiction, it is critical to extend this application to address the unique challenges posed by drug abuse.

Mackintosh and Stewart \cite{Mackintosh1979}, demonstrated the applicability of classical epidemic models in understanding the dynamics of drug-addiction and formulating effective control strategies. 
Building on this work, White and Comiskey \cite{White2007} proposed an ODE mathematical model that categorizes the population into susceptible individuals ($S$), drug users not in treatment ($D_1$), and drug users in treatment ($D_2$).
The researchers defined the basic reproduction number $\mathcal{R}_0$ and performed sensitivity analyses. 
In addition, they investigated the existence of backward bifurcation at the drug-addiction equilibrium when $\mathcal{R}_0 = 1$ and concluded that preventive measures are more effective than curative interventions in combating drug abuse. 
In addition, several delay differential equations (DDEs) and fractional differential equations (FDEs) models have been proposed and studied in \cite{Kundu2021, Zhang2019, Zinihi2024MM, Zinihi2024OC}, with a focus on global dynamics and Hopf bifurcation within these models.
This approach not only expands the application of mathematical modeling to drug-addiction, 
but also underscores the importance of tailored interventions and prevention strategies to mitigate the societal impact of drug abuse.

However, the models discussed above predominantly operate in a homogeneous environment. 
Recognizing the importance of host population mobility and spatial heterogeneity in disease spread modeling, it becomes imperative to consider these factors for a more complete understanding. 
A notable example is the work of \cite{SidiAmmi2023}, who introduced a diffusive SIR epidemic model and explored the effects of spatial heterogeneity and vaccination strategies on disease persistence and extinction. 
Similarly, \cite{Song2019} explored a SEIRS model in a heterogeneous environment, highlighting the substantial influence of movement and spatial heterogeneity among exposed and recovered individuals on disease dynamics. 
Additional insights into such models can be found in the literature \cite{Yang2019, Wang2020}.

In this paper, we propose and analyze a cocaine-heroin model with spatial distributions, based on reaction-diffusion equations that describe how substances (such as drugs) spread and interact in space.
Here, we describe the dynamics of cocaine and heroin concentrations in a population, and the spatial distributions refer to how these concentrations vary across locations (e.g., neighborhoods, cities). 
The reaction terms capture interactions such as drug use, relapse, and potential immunization effects, and the diffusion coefficients determine how quickly the drugs spread through space.
Such models are used to study drug epidemics, predict the spread of addiction, and the impact of public health interventions in specific areas.
By analyzing the stability and behavior of drug-free equilibria, they can assess the effectiveness of control measures.

Given the recognized importance of spatial heterogeneity in disease spread modeling, its inclusion is considered essential. 
To address the nuances of heroin spread in the context of spatial heterogeneity, inspired by the model proposed by \cite{White2007}, \cite{Duan2020} introduced and examined an age-structured heroin epidemic model incorporating reaction-diffusion dynamics. 
The threshold dynamics of this model were examined in terms of the basic reproduction number $\mathcal{R}_0$. 
Another paper by \cite{Wang2020} presented a heroin epidemic model incorporating reaction-diffusion dynamics. 
Using next-generation operator theory, they defined the basic reproduction number $\mathcal{R}_0$ as a critical threshold that determines the extinction or persistence of the heroin epidemic. 
Further insights into the dynamics, especially for the critical case where $\mathcal{R}_0$ equals 1, were refined by \cite{Xu2021}.

Previous research on dynamic models of drug abuse has focused primarily on single substances, particularly heroin. 
However, it is important to recognize that cocaine and heroin users represent overlapping populations, with many individuals using both substances simultaneously \cite{Harrell2012}.
These significant bidirectional influences were also observed in a study in Baltimore, Maryland \cite{bohnert2009social} that demonstrated the need for joint modeling of cocaine and heroin addiction. The follow-up study \cite{Sauer2021} underlined the importance of spatio-temporal modeling.
It is worth mentioning that Benneyan et al.\ \cite{benneyan2017modeling} proposed an ODE model for the opioid and heroin co-epidemic crisis.
In examining the interplay between the cocaine and heroin epidemics, \cite{Chekroun2017} introduced and analyzed an age-structured model that included both drugs. 
The study explored threshold dynamics in relation to the basic reproduction number. 
Notably, spatial heterogeneity and host population movement were not considered in \cite{Chekroun2017}.
To our knowledge, there is a notable gap in the literature regarding models of cocaine-heroin coinfection that incorporate population density variation over both time and spatial variables.
Addressing this gap is critical for a more nuanced understanding of the complex dynamics of cocaine-heroin co-epidemics, and highlights the need for comprehensive models that account for both substance interactions and the spatio-temporal distribution of the affected population.

The incorporation of the Laplacian operator into the cocaine-heroin mathematical model serves as a crucial motivation to capture the spatial diffusion dynamics inherent in the spread of the cocaine-heroin epidemic. 
In this model, the random movement of individuals through space is a significant factor contributing to the spread of the epidemic. 
The diffusion coefficient ($d > 0$) assigned to each compartment $S$ (susceptible), $C$ (cocaine users), $H$ (heroin users), and $R$ (recovered) reflects the different rates at which individuals move within the spatial domain. 
This operator provides a mathematical representation of how the concentration of each compartment changes across space, incorporating the spatial interactions that influence the spread of the cocaine-heroin epidemic.
This model allows for a holistic examination of the complex dynamics involved in the spatial spread of the epidemic. 
By considering spatial interactions and individual movements, the model provides valuable insights into the spatial aspects of disease transmission.
As a result, the enhanced cocaine-heroin model contributes to a more thorough understanding of the spatial spread of the epidemic. 
This understanding is essential for informed decision-making and the development of proactive measures in public health planning and response. 
The incorporation of the Laplacian operator enriches the mathematical representation of the cocaine-heroin epidemic, making the model a powerful tool for exploring and addressing the spatial dimensions of this public health challenge.

This study is organized as follows. In Section~\ref{S2}, we introduce a cocaine-heroin SCHR model that incorporates essential concepts related to the Laplacian operator. 
This section lays the foundation for the analyses that follow. 
In addition, Section~\ref{S3} systematically establishes the existence, uniqueness, and boundedness of a globally non-negative solution to our epidemiological problem, thus strengthening the theoretical framework of the investigation. 
Sections~\ref{S4}, \ref{S5}, and \ref{S6} are devoted to a qualitative evaluation of the spatial model, emphasizing the determination of stability properties at both local and global scales. 
In Section~\ref{S7}, we extend the proposed cocaine-heroin model by including two compartments for cocaine and heroin users in treatment. A corresponding analysis of the extended model is presented.
The paper includes a series of compelling numerical experiments in Section~\ref{S8}. 
Finally, Section~\ref{S9} summarizes this work with a comprehensive summary of the main results.


\section{Mathematical model}\label{S2}
Let $\mathcal{U}\subset \mathbb{R}^2$ be a bounded domain with a smooth boundary denoted by $\partial \mathcal{U}$. 
Let $b\in\mathbb{R}^*_+$. 
The transmission coefficients of the cocaine-heroin model are given in Table~\ref{Tab1}.
In contrast, Figure~\ref{F1} contains the transfer diagram for the proposed model.

\begin{table}[hbtp]
\centering
\setlength{\tabcolsep}{0.5cm}
\caption{Transmission coefficients for the SCHR model.}\label{Tab1}
\begin{tabular}{cc}
\hline 
\quad Symbol \quad & \quad Description \quad \\
\hline \hline 
$\eta_i$ & Death rates of $S, H, C,$ and $R$\\ 
\hline
$d$ & Diffusion rate of $S, H, C,$ and $R$\\
\hline
$\gamma_j$ & Natural recovery rates of $H$ and $C$\\
\hline
$\sigma$ & Recovery rate of $C$ to $H$\\
\hline
$\Lambda$ & Recruitment rate of the population\\
\hline
$\beta$ & Transmission rate\\ 
\hline
\end{tabular}
\end{table}

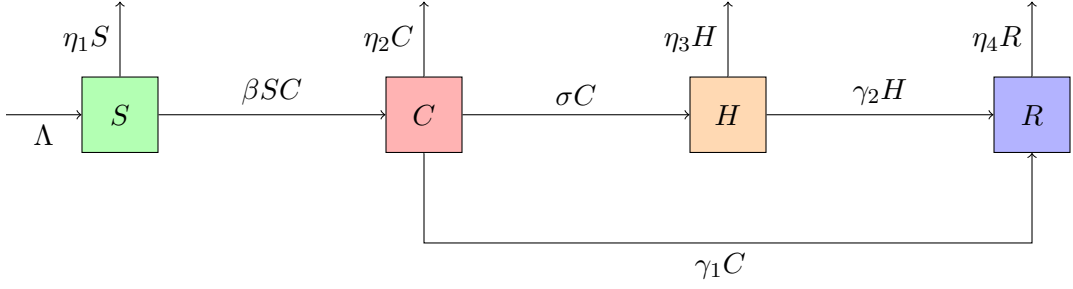
\begin{figure}[hbtp]
\centering
\begin{tikzpicture}[node distance=4cm]
\node (S) [rectangle, draw, minimum size=1cm, fill=green!30] {$S$}; 
\node (C) [rectangle, draw, minimum size=1cm, fill=red!30, right of=S] {$C$};
\node (H) [rectangle, draw, minimum size=1cm, fill=orange!30, right of=C] {$H$};
\node (R) [rectangle, draw, minimum size=1cm, fill=blue!30, right of=H] {$R$}; 
\draw[->] (-1.5,0) -- ++(S) node[midway,below]{$\Lambda$};
\draw [->] (S) -- (C) node[midway,above]{$\beta SC$};
\draw [->] (C) -- (H) node[midway,above]{$\sigma C$};
\draw [->] (H) -- (R) node[midway,above]{$\gamma_2 H$};
\draw [->] (C.south) -| (4,-1.7) -- (11.8,-1.7) node[midway,below]{$\gamma_1 C$} -| (R.south);
\draw[->] (S.north) -| (0,1.5) node[near end,left]{$\eta_1 S$};
\draw[->] (C.north) -| (4,1.5) node[near end,left]{$\eta_2 C$};
\draw[->] (H.north) -| (8,1.5) node[near end,left]{$\eta_3 H$};
\draw[->] (R.north) -| (12,1.5) node[near end,left]{$\eta_4 R$};
\end{tikzpicture}
\captionof{figure}{Transfer diagram for the proposed SCHR model.}\label{F1}
\end{figure}

Given the assumptions outlined above, the reaction-diffusion cocaine-heroin model can be formulated as follows
\begin{equation}\label{E2.1}
\left\{\begin{aligned}
\partial_t S - d \Delta S &= \Lambda - \beta S C - \eta_1 S, \\
\partial_t C - d \Delta C &= \beta S C - (\eta_2+\sigma+\gamma_1) C,\\ 
\partial_t H - d \Delta H &= -(\eta_3+\gamma_2) H + \sigma C,\\ 
\partial_t R - d \Delta R &= \gamma_1 C + \gamma_2 H - \eta_4 R,
\end{aligned}\right. \ \text{ in } \Omega_b = [0, b]\times\mathcal{U}, 
\end{equation}
with the no-flux boundary conditions ($\nu$ denote the outer unit normal) and the initial conditions
\begin{equation}\label{E2.2} 
\left\{\begin{aligned}
&\frac{\partial S}{\partial \nu} = \frac{\partial C}{\partial \nu} = \frac{\partial H}{\partial \nu} = \frac{\partial R}{\partial \nu} = 0, \ \text{ on } \Sigma_b = [0, b]\times\partial\mathcal{U},\\
&S(0,x)=S^0, \ C(0,x)=C^0, \ H(0,x)=H^0, \ R(0,x)=R^0, \ \text{ in } \mathcal{U}.
\end{aligned}\right.
\end{equation}

Let $\psi = (\psi_1, \psi_2, \psi_3, \psi_4) = (S, C, H, R)$, 
$\psi^0 =( \psi^0_1 ,\psi^0_2 , \psi^0_3, \psi^0_4)=(S^0, C^0, H^0, R^0)$, 
$\mathbb{X}(\mathcal{U}) = \bigl(L^2(\mathcal{U})\bigr)^4$, and $\mathcal{B}$ defined as follows
\begin{equation}\label{E2.3}
   \mathcal{B}\colon \ 
\begin{aligned}
&\mathcal{D}_\mathcal{B} = \Big\{\upsilon \in \bigl( H^2(\mathcal{U}) \bigr)^4 \mid \ \frac{\partial \upsilon_i}{\partial \nu} = 0,\; 1\leq i\leq 4\Big\} \subset \mathbb{X}(\mathcal{U}) \longrightarrow \mathbb{X}(\mathcal{U}),\\ 
&\upsilon \longrightarrow -d\Delta \upsilon = (-d\Delta \upsilon_i)_{i=1,2,3,4}.
\end{aligned}
\end{equation}
We introduce the function $\xi$, which is defined by
\begin{equation*}
   \xi(\psi(t))=\bigl(\xi_1(\psi(t)),\,\xi_2(\psi(t)),\,\xi_3(\psi(t)),\,\xi_4(\psi(t))\bigr), \ t\in [0, b],
\end{equation*}
with $\psi(t)(\cdot) = \psi(t,\cdot)$, and 
\begin{equation*}
\begin{cases}
  \xi_1(\psi(t)) = \Lambda - \beta \psi_1\psi_2 - \eta_1 \psi_1,\\
  \xi_2(\psi(t))= \beta \psi_1\psi_2 - (\eta_2 + \sigma + \gamma_1)\psi_2,\\
  \xi_3(\psi(t)) = - (\eta_3 + \gamma_2)\psi_3 + \sigma \psi_2,\\
  \xi_4(\psi(t))= \gamma_1\psi_2 + \gamma_2\psi_3 - \eta_4\psi_4.
\end{cases}
\end{equation*}
Subsequently, \eqref{E2.1}--\eqref{E2.2} can be reformulated in  $\mathbb{X}( \mathcal{U})$ as follows
\begin{equation}\label{E2.4}
\begin{cases}
\partial_t \psi(t) + \mathcal{B}\psi(t) = \xi(\psi(t)),\\
\psi(0)=\psi^0.
\end{cases} \ t\in[0, b],
\end{equation}


\section{Existence of the strong solution}\label{S3}
Let $(t, x) \in \Omega_b$. The function $\xi$ exhibits Lipschitz continuity in $\psi$ uniformly with respect to $t \in[0, b]$. Since $-\Delta$ is obviously strongly elliptic (\cite{Pazy1983}), 
according to \cite{Barbu1994, Pazy1983}, the system \eqref{E2.4} admits a unique strong solution $\psi\in W^{1,2}(0, T ; \mathbb{X}(\mathcal{U}))$ with
\begin{equation*}
\psi_i \in L^2(0, T ; H^2(\mathcal{U})) \cap L^\infty(0, T ; H^1(\mathcal{U})), \ \forall i\in\{1, 2, 3, 4\}.
\end{equation*}
This result allows us to establish the following theorem.

\begin{theorem}
The problem \eqref{E2.1}--\eqref{E2.2} admits a unique non-negative global solution $\psi\in W^{1,2} (0, T ; \mathbb{X}(\mathcal{U}))$. 
Furthermore, for each $i\in\{1, 2, 3, 4\}$ we get
\begin{equation*}
   \psi_i\in L^2(0, T ; H^2(\mathcal{U})) \cap L^\infty(0, T ; H^1(\mathcal{U})) \cap L^\infty(\Omega_b).
\end{equation*}
\end{theorem}

\begin{proof}
We want to show the boundedness of $\psi$ in the setting of $\Omega_b$. Let $\{\mathcal{S}_1(t), t \geq 0\}$ be the $C_0$ semigroup generated by the operator component $\mathcal{B}_1$, where $\mathcal{B}_1 \psi = d\Delta \psi_1$, and
\begin{equation*}
   \delta_1 = \max\Bigl\{ \bigl\|\xi_1\bigr\|_{L^{\infty}(\Omega_b))}, \bigl\|\psi_1^0\bigr\|_{L^{\infty}(\mathcal{U})} \Bigr\}.
\end{equation*}
It is apparent that the function
\begin{equation*}
   \phi_1(t, x)=\psi_1- \delta_1 t-\bigl\|\psi_1^0\bigr\|_{L^{\infty}(\mathcal{U})},
\end{equation*}
satisfies the Cauchy problem
\begin{equation}\label{E3.1}
\begin{cases}
  \partial_t \phi_1 (t, x)=d \Delta \phi_1+\xi_1(\psi(t))- \delta_1,  \\
  \phi_1(0, x) = \psi_1^0 -\bigl\|\psi_1^0\bigr\|_{L^{\infty}(\mathcal{U})},
\end{cases} \ t \in[0, b].
\end{equation}
The strong solution of \eqref{E3.1} is given by
\begin{equation*}
\phi_1(t)= \mathcal{S}_1(t) \Bigl(\psi_1^0 - \bigl\|\psi_1^0\bigr\|_{L^{\infty}(\mathcal{U})} \Bigr) 
+\int_0^t \mathcal{S}_1(t-s) \bigl(\xi_1(\psi(t))- \delta_1\bigr) \,ds,
\end{equation*}
Since $\psi_1^0-\bigl\|\psi_1^0\bigr\|_{L^{\infty}(\mathcal{U})} \leq 0$ and $\xi_1(\psi(t))- \delta_1 \leq 0$, 
it follows that $\phi_1(t, x) \leq 0$.
Similarly, the function 
\begin{equation*}
   \Tilde{\phi}_1(t, x) = \psi_1+ \delta_1 t + \bigl\|\psi_1^0\bigr\|_{L^{\infty}(\mathcal{U})},
\end{equation*}
satisfies the Cauchy problem
\begin{equation}\label{E3.2}
\begin{cases}
\partial_t \Tilde{\phi}_1 (t, x)=d \Delta \Tilde{\phi}_1 + \xi_1(\psi(t))+ \delta_1,\\
\Tilde{\phi}_1(0, x)=\psi_1^0 + \bigl\|\psi_1^0\bigr\|_{L^{\infty}(\mathcal{U})}.
\end{cases} \ t \in[0, b],
\end{equation}
The strong solution of \eqref{E3.2} is expressed as
\begin{equation*}
\Tilde{\phi}_1(t)= \mathcal{S}_1(t) \Bigl(\psi_1^0 - \bigl\|\psi_1^0\bigr\|_{L^{\infty}(\mathcal{U})} \Bigr) 
+ \int_0^t \mathcal{S}_1(t-s) \bigl(\xi_1(\psi(t))+ \delta_1\bigr) \,ds.
\end{equation*}
Since $\psi_1^0 + \bigl\|\psi_1^0\bigr\|_{L^{\infty}(\mathcal{U})} \geq 0$ and $\xi_1(y(t))+ \delta_1 \geq 0$, it follows that $\Tilde{\phi}_1(t, x) \geq 0$. Consequently,
\begin{equation*}
   \bigl|\psi_1(t, x)\bigr| \leq \delta_1 t + \bigl\|\psi_1^0\bigr\|_{L^{\infty}(\mathcal{U})},
\end{equation*}
and analogously 
\begin{equation*}
   \bigl|\psi_j(t, x)\bigr| \leq \delta_1 t + \bigl\|\psi_j^0\bigr\|_{L^{\infty}(\mathcal{U})}, \text { for } j = 2, 3, 4.
\end{equation*}
Thus, we have established  that $\psi_i \in L^{\infty}(\Omega_b)$.

Let's establish the non-negativity of $\psi$. 
Starting with $\psi_4$, by expressing $\psi_4 = \psi_4^+ - \psi_4^-$, where
\begin{equation*}
   \psi_4^+(t, x) = \sup\{\psi_2(t, x), 0\} \text{ and } \psi_4^-(t, x) = \sup\{-\psi_4(t, x), 0\}.
\end{equation*}
If we multiply the equation corresponding to $i=2$ in \eqref{E2.4} with $\psi_4^-$, we get
\begin{equation*}
   -\frac{1}{2} \frac{d}{d t} \bigl\|\psi_4^-\bigr\|_{L^2(\mathcal{U})}^2
   = d\int_{\mathcal{U}}\bigl|\nabla \psi_4^-\bigr|^2 \,dx
     - \gamma_1\int_{\mathcal{U}} \psi_2 \psi_4^- \,dx 
     - \gamma_2\int_{\mathcal{U}} \psi_3 \psi_4^- \,dx 
     + \eta_4\int_{\mathcal{U}} (\psi_4^-)^2 \,dx,
\end{equation*}
which can be rewritten as
\begin{equation*}
    \frac{1}{2} \frac{d}{d t} \bigl\|\psi_4^-\bigr\|_{L^2(\mathcal{U})}^2 \leq \gamma_1 \int_{\mathcal{U}} \psi_2 \psi_4^- \,dx
    + \gamma_2\int_{\mathcal{U}} \psi_3 \psi_4^- \,dx .
\end{equation*}
Using the Cauchy-Schwarz inequality, we get
\begin{equation*}
   \frac{d}{d t} \bigl\|\psi_4^-\bigr\|_{L^2(\mathcal{U})}^2 
   \leq C (\bigl\|\psi_2\bigr\|_{L^2(\mathcal{U})}^2 
   + \bigl\|\psi_3\bigr\|_{L^2(\mathcal{U})}^2) 
     \bigl\|\psi_4^-\bigr\|_{L^2(\mathcal{U})}^2.
\end{equation*}
Recall that $\psi_i \in L^{\infty}(\Omega_b)$, applying Gronwall's inequality, we get
\begin{equation*}
    \bigl\|\psi_4^-\bigr\|_{L^2(\mathcal{U})}^2 \leq 0,
\end{equation*}
leading to the conclusion that $\psi_4^- = 0$. 
Consequently, $\psi_4 \geq 0$ in $\Omega_b$.

Using a similar methodology applied to $\psi_2$, we derive the following expression
\begin{equation*}
-\frac{1}{2} \frac{d}{d t}\bigl\|\psi_2^-\bigr\|_{L^2(\mathcal{U})}^2 =    
  d\int_{\mathcal{U}}\bigl|\nabla \psi_2^-\bigr|^2 \,dx 
 - \beta \int_{\mathcal{U}} \psi_1 (\psi_2^-)^2 \,dx 
 + (\eta_2 + \sigma + \gamma_1)\int_{\mathcal{U}} (\psi_2^-)^2 \,dx,
\end{equation*}
which can be expressed as
\begin{equation*}
\frac{1}{2} \frac{d}{d t} \bigl\|\psi_2^-\bigr\|_{L^2(\mathcal{U})}^2 
\leq \beta \int_{\mathcal{U}} \psi_1 (\psi_2^-)^2 \,dx
\leq \beta N \bigl\|\psi_2^-\bigr\|_{L^2(\mathcal{U})}^2.
\end{equation*}
Applying Gronwall's inequality, we obtain
\begin{equation*}
   \bigl\|\psi_2^-\bigr\|_{L^2(\mathcal{U})}^2 \leq 0,
\end{equation*}
which implies $\psi_2^- = 0$. Then $\psi_2 \geq 0$.

Let
\begin{align*}
  G_1(\psi_1, \psi_3) &= \Lambda - \beta \psi_1\psi_2 - \eta_1 \psi_1,\\
  G_3(\psi_1, \psi_3) &= - (\eta_3 + \gamma_2)\psi_3 + \sigma \psi_2.
\end{align*}
To move to the next step in our argument, we turn our attention to the system
\begin{equation}\label{E3.3}
\begin{cases}
  \partial_t \psi_j = d \Delta \psi_j + G_j(\psi_1, \psi_3), \vspace{0.15cm}\\
  \psi_j(0)=\psi_j^0,
\end{cases} \ j=1,3.
\end{equation}
It is obvious that these functions are continuously differentiable, with $G_1(0, \psi_3) = \Lambda$ and $G_3(\psi_1, 0) = \sigma_1\psi_2$ for all $\psi_1, \psi_3$. 
Given the non-negativity of the initial conditions of \eqref{E3.3}, and in accordance with \cite{Smoller2012}, we establish the non-negativity of $\psi_1$ and $\psi_3$ in $\Omega_b$.
\end{proof}

\begin{theorem}
Let $\psi$ be the solution of \eqref{E2.1}--\eqref{E2.2}. Then
\begin{equation*}
   \Bigl\|\frac{\partial \psi_i}{\partial t}\Bigr\|_{L^2(\Omega_b)}
   +\bigl\|\psi_i\bigr\|_{L^2(0, T ; H^2(\mathcal{U}))}
   +\bigl\|\psi_i\bigr\|_{H^1(\mathcal{U})} 
   +\bigl\|\psi_i\bigr\|_{L^{\infty}(\Omega_b)} < \infty.
\end{equation*}
\end{theorem}

\begin{proof}
The first equation of \eqref{E2.4} gives
\begin{equation*}
\begin{aligned}
\int_0^t \int_{\mathcal{U}} \Bigl|\frac{\partial \psi_1}{\partial  \tau}\Bigr|^2 \,d\tau dx 
- 2 d \int_0^t \int_{\mathcal{U}} \frac{\partial \psi_1}{\partial \tau} \Delta \psi_1 \,d\tau dx 
&+ d^2 \int_0^t \int_{\mathcal{U}} |\Delta \psi_1|^2 \,d\tau dx \\
& =\int_0^t \int_{\mathcal{U}} (\Lambda - \beta \psi_1\psi_2 - \eta \psi_1)^2 \,d\tau dx.
\end{aligned} 
\end{equation*}
Because of
\begin{equation*}
  \int_0^t \int_{\mathcal{U}} \frac{\partial \psi_1}{\partial \tau} \Delta \psi_1 \,d\tau dx
    = \int_{\mathcal{U}}\Bigl(-\bigl|\nabla \psi_1\bigr|^2 + \bigl|\nabla \psi_1^0\bigr|^2\Bigr) \,dx,
\end{equation*}
we have
\begin{equation*}
\begin{aligned}
\int_0^t \int_{\mathcal{U}} \Bigl|\frac{\partial \psi_1}{\partial \tau}\Bigr|^2 \,d\tau dx
+ d^2 \int_0^t \int_{\mathcal{U}} \bigl|\Delta \psi_1\bigr|^2 \,d\tau dx 
& + 2 d \int_{\mathcal{U}} \bigl|\nabla \psi_1\bigr|^2 \,dx
-2 d \int_{\mathcal{U}} \bigl|\nabla \psi_1^0\bigr|^2 \,dx \\
& =\int_0^t \int_{\mathcal{U}} (\Lambda - \beta \psi_1\psi_2 - \eta \psi_1)^2 \,d\tau dx.
\end{aligned}
\end{equation*}
Due to the boundedness of $\psi_1^0 \in H^2(\mathcal{U})$ and $\|\psi_i\|_{L^{\infty}(Q)}$, the result holds for i = 1. 
Analogous reasoning can be applied to the remaining scenarios.
\end{proof}


\section{Equilibria of the proposed cocaine-heroin model}\label{S4}
To identify the equilibrium points, we set the right side of the system \eqref{E2.1} equal to zero.
This results in two equilibria in the coordinate space $(S, C, H, R)$. More precisely, $E_f=(\frac{\Lambda}{\eta_1}, 0, 0, 0)$ denotes the drug-free equilibrium, while $E^*=(S^*, C^*, H^*, R^*)$ denotes the drug-addiction equilibrium.

The Jacobian matrix of \eqref{E2.1} at the point $(S, C, H, R)$, excluding diffusion effects, is expressed as follows
\begin{equation*}
  \mathcal{J} = \begin{pmatrix}
  -\eta_1-\beta C & -\beta S & 0 & 0\\
  \beta C & \beta S-(\eta_2+\sigma+\gamma_1) & 0 & 0\\
  0 & \sigma & -(\eta_3+\gamma_2) & 0\\
  0 & \gamma_1 & \gamma_2 & -\eta_4
\end{pmatrix}.
\end{equation*}
This matrix $\mathcal{J}$ is decomposed into a sum of two matrices; the transmission matrix $\mathcal{T}$ and the transition matrix $\mathcal{K}$\footnote{$\mathcal{T}$ accounts for the number of new infections, while $\mathcal{K}$ is used to characterize movement between compartments.}, where
\begin{equation*}
   \mathcal{T} = \begin{pmatrix}
   -\beta C & -\beta S & 0 & 0\\
   \beta C & \beta S & 0 & 0\\
   0 & 0 & 0 & 0\\
   0 & 0 & 0 & 0
\end{pmatrix},
\end{equation*}
and
\begin{equation*}
\mathcal{K} = \begin{pmatrix}
-\eta_1 & 0 & 0 & 0\\
0 & -(\eta_2+\sigma+\gamma_1) & 0 & 0\\
0 & \sigma & -(\eta_3+\gamma_2) & 0\\
0 & \gamma_1 & \gamma_2 & -\eta_4
\end{pmatrix}.
\end{equation*}
Therefore, at the drug-free equilibrium point $E_f$, the basic reproduction number of \eqref{E2.1} is given by
\begin{equation*}
    \mathcal{R}_0 = \operatorname{trace}(-\mathcal{T}\mathcal{K}^{-1}) 
    = \frac{\beta \Lambda}{\eta_1 (\eta_2 + \sigma + \gamma_1)}.
\end{equation*}
In this case, we have
\begin{equation*}
    S^* = \frac{\Lambda}{\eta_1 \mathcal{R}_0}, 
    \quad C^* = \frac{\eta_1}{\beta} (\mathcal{R}_0 - 1), 
    \quad H^* = \frac{\eta_1\sigma }{\beta(\eta_3 + \gamma_2)} (\mathcal{R}_0 - 1), 
    \quad R^* = \frac{\eta_1}{\beta\eta_4} \Bigl(\gamma_1 + \frac{\sigma\gamma_2}{\eta_3 + \gamma_2}\Bigr) (\mathcal{R}_0 - 1).
\end{equation*}
It is noteworthy that the variable $R$ is absent from the first three equations in \eqref{E2.1}. 
This allows us to focus our analysis on the following reduced model
\begin{equation}\label{E4.1}
   \left\{\begin{aligned}
&\begin{aligned}
&\partial_t S - d \Delta S = \Lambda - \beta S C - \eta_1 S, \\
&\partial_t C - d \Delta C = \beta S C - (\eta_2+\sigma+\gamma_1) C,\\ 
&\partial_t H - d \Delta H = -(\eta_3+\gamma_2) H + \sigma C,
\end{aligned} \quad \text{ in } \Omega_b,\\
&\frac{\partial S}{\partial \nu} = \frac{\partial C}{\partial \nu} = \frac{\partial H}{\partial \nu} = 0, \quad \text{ on } \Sigma_b,\\ \vspace{0.15cm}
&\bigl(S(0,x), C(0,x), H(0,x)\bigr) = \bigl(S^0, C^0, H^0\bigr), \quad \text{ in } \mathcal{U}.
\end{aligned}\right.
\end{equation}


\section{Local stability of the equilibria}\label{S5}
Using the conventional approach of \cite{Pang2004}, consider $0=\lambda_1<\lambda_2<\lambda_3<\dots$ as the eigenvalues of the operator $-\Delta$ on $\mathcal{U}$, subject to homogeneous Neumann boundary conditions.

Let $\mathcal{H}(\mathcal{U}) = \bigl\{y=(S, C, H) \in[C(\bar{\mathcal{U}})]^3 \mid \frac{\partial S}{\partial \nu} = \frac{\partial C}{\partial \nu} = \frac{\partial H}{\partial \nu} = 0, \text{ on } \partial \mathcal{U}\bigr\}$, and let $\mathcal{H}_j$ represent the invariant subspace of $\mathcal{H}(\mathcal{U})$ associated with a given eigenvalue $\lambda_j$. 
This forms a direct sum $\oplus_{j=1}^{\infty} \mathcal{H}_j$.
The Jacobian matrix of \eqref{E4.1} at $(S, C, H)$, without diffusion, is given by
\begin{equation*}
  \mathcal{J} = \begin{pmatrix}
  -\eta_1-\beta C & -\beta S & 0\\
  \beta C & \beta S-(\eta_2+\sigma+\gamma_1) & 0\\
  0 & \sigma & -(\eta_3+\gamma_2)
\end{pmatrix}.
\end{equation*}
The linearization of \eqref{E4.1} is described by $\partial_t y = \mathcal{N} y$, where
\begin{equation*}
\mathcal{N} = \begin{pmatrix}
-d \lambda_j & 0 & 0\\
0 & -d \lambda_j & 0\\
0 & 0 & -d \lambda_j 
\end{pmatrix} + 
\begin{pmatrix}
-\eta_1-\beta C & -\beta S & 0\\
\beta C & \beta S-(\eta_2+\sigma+\gamma_1) & 0\\
0 & \sigma & -(\eta_3+\gamma_2)
\end{pmatrix}.
\end{equation*}
Then, the characteristic equation is given by
\begin{equation}\label{E5.1}
   ( X + d\lambda_j + \eta_3+\gamma_2 )
   \Bigl[ (X + d\lambda_j + \eta_1 + \beta C) \bigl(X + d\lambda_j + (\eta_2+\sigma+\gamma_1) -\beta S\bigr) + \beta^2 S C\Bigr] = 0.
\end{equation}
The local stability analysis of the drug-free equilibrium $E_f$ and the drug-addiction equilibrium $E^*$ can be established through the following theorem

\begin{theorem}
The drug-free equilibrium $E_f$ is locally asymptotically stable if $\mathcal{R}_0 < 1$ and unstable if $\mathcal{R}_0 > 1$. 
While the drug-addiction equilibrium $E^*$ is locally asymptotically stable if $\mathcal{R}_0 > 1$.
\end{theorem}

\begin{proof}
Evaluating \eqref{E5.1} at $E_f$, the equation becomes
\begin{equation}\label{E5.2}
     ( X + d\lambda_j + \eta_3+\gamma_2 ) (X + d\lambda_j + \eta_1 )
     \Bigl(X + d\lambda_j + (\eta_2+\sigma+\gamma_1) -\beta \frac{\Lambda}{\eta_1}\Bigr) = 0.
\end{equation}
    The roots of the equation \eqref{E5.2} are identified as $X_1 = -d\lambda_j - \eta_1$, 
    $X_2 = -d\lambda_j - (\eta_2+\sigma+\gamma_1) + \beta \frac{\Lambda}{\eta_1} = -d\lambda_j - (\eta_2+\sigma+\gamma_1)(1 - \mathcal{R}_0)$,
    and $X_3 = -d\lambda_j - (\eta_3+\gamma_2)$. 
    Therefore, the drug-free equilibrium $E_f$ is locally asymptotically stable if $\mathcal{R}_0 < 1$ and unstable if $\mathcal{R}_0 > 1$.
    
After evaluating \eqref{E5.1} at $E^*$, the characteristic equation can be expressed as
\begin{equation}\label{E5.3}
   (X + d\lambda_j)^2 + \alpha_1(X + d\lambda_j) + \alpha_2 =0,
\end{equation}
with
\begin{equation*}
   \alpha_1 = \eta_1 + (\eta_2 + \sigma + \gamma_1) + \beta (C^*-S^*) 
   \text{ and } 
   \alpha_2 = (\eta_1 + \beta C^*)(\eta_2+\sigma+\gamma_1) - \eta_1 \beta S^*.
\end{equation*}
In the scenario where $\mathcal{R}_0 > 1$, it can be deduced that $\alpha_1>0$ and $\alpha_2>0$. 
This implies that all roots of the characteristic equation \eqref{E5.3} have negative real parts, cf.\ \cite{Henry1981}. 
Thus, $E^*$ is locally asymptotically stable.
\end{proof}


\section{Global stability of the equilibria}\label{S6}
We turn our attention to the global stability analysis of the reaction-diffusion equations \eqref{E2.1}--\eqref{E2.2}, by following the approach outlined in \cite{Hattaf2013}.
We note that a global stability analysis for a pure heroin model was done by Zhang and Xing \cite{zhang2020stability}.

Considering the solution $\psi$ of \eqref{E2.1}--\eqref{E2.2}, we define a Lyapunov functional at $E_f$ for the associated ODE system as $\mathcal{F}_1(\psi) = C$. 
The time derivative of $\mathcal{F}_1$ is given by
\begin{equation*}
\begin{split}
\frac{d \mathcal{F}_1}{d t} 
& = \nabla \mathcal{F}_1(\psi) \cdot \xi(\psi)\\
& = \beta S C - (\eta_2+\sigma+\gamma_1) C 
= \bigl(\beta S - (\eta_2+\sigma+\gamma_1)\bigr)C\\
& \leq \Bigl(\beta \frac{\Lambda}{\eta_1} - (\eta_2+\sigma+\gamma_1)\Bigr)C\\ 
&= (\eta_2+\sigma+\gamma_1)\Bigl(\beta\, \frac{\Lambda}{\eta_1(\eta_2+\sigma+\gamma_1)} - 1\Bigr)C\\
& = (\eta_2+\sigma+\gamma_1) (\mathcal{R}_0 - 1)C
\end{split}
\end{equation*}
Following the methodology outlined in \cite{Hattaf2013}, we establish a Lyapunov functional for the system \eqref{E2.1}--\eqref{E2.2} at the drug-free equilibrium $E_f$ as
\begin{equation*}
    \mathcal{G}_1=\int_{\mathcal{U}} \mathcal{F}_1(\psi(t, x)) d x.
\end{equation*}
By computing the time derivative of $\mathcal{G}_1$ along the non-negative solution of \eqref{E2.1}--\eqref{E2.2}, we obtain
\begin{equation*}
    \frac{d \mathcal{G}_1}{d t} 
    \leq \int_{\mathcal{U}} (\eta_2+\sigma+\gamma_1) (\mathcal{R}_0 - 1)C \,dx.
\end{equation*}
We introduce a nonlinear Lyapunov function at $E^*$ of the associated ODE of \eqref{E2.1}--\eqref{E2.2}, as
\begin{equation*}
     \mathcal{F}_2(\psi) =  \Bigl(S-S^*-S^* \ln \frac{S}{S^*}\Bigr) 
     + \Bigl(C-C^*-C^* \ln \frac{C}{C^*}\Bigr).
\end{equation*}
The time derivative of this function can be expressed as
\begin{equation}\label{E6.1}
\begin{split}
\frac{d \mathcal{F}_2}{d t} &=  
\Bigl((\Lambda - \beta S C - \eta_1 S)-\frac{S^*(\Lambda - \beta S C - \eta_1 S)}{S}\Bigr) \\
&q\quad +\Bigl((\beta S C-(\eta_2+\sigma+\gamma_1) C)-\frac{C^*(\beta S C-(\eta_2+\sigma+\gamma_1) C)}{C}\Bigr) .
\end{split}
\end{equation}
Because of
\begin{equation}\label{E6.2}
\Lambda = \beta S^* C^* + \eta_1 S^* \ \text{ and } \ \quad \eta_2 + \sigma + \gamma_1 = \beta S^*.
\end{equation}
Then, substituting \eqref{E6.2} from \eqref{E6.1}, we obtain the following expression
\begin{equation*}
   \frac{d \mathcal{F}_2}{d t} = -\frac{\eta_1}{S} (S-S^*)^2
   -\frac{\beta C^*}{S}(S-S^*)^2.
\end{equation*}
Using similar techniques outlined in \cite{Hattaf2013}, we construct a Lyapunov functional for the system \eqref{E2.1}--\eqref{E2.2} at the epidemic equilibrium $E^*$, defined as
\begin{equation*}
   \mathcal{G}_2 = \int_{\mathcal{U}} \mathcal{F}_2(\psi(t, x)) \,dx.
\end{equation*}
By computing the time derivative of $\mathcal{G}_2$ along the non-negative solution of the model \eqref{E2.1}--\eqref{E2.2}, we obtain 
\begin{equation*}
    \frac{d \mathcal{G}_2}{d t} 
    = \int_{\mathcal{U}} \Bigl( -\frac{\eta}{S}(S-S^*)^2 
    -\frac{\beta C^*}{S}(S-S^*)^2 \Bigr) \,dx 
    - d \Bigl( S^* \int_{\mathcal{U}} \frac{|\nabla S|^2}{S^2}dx + C^* \int_{\mathcal{U}} \frac{|\nabla C|^2}{C^2}dx \Bigr).
\end{equation*}

According to the LaSalle's Invariance Principle \cite{LaSalle1976}, 
we can summarize the results of global stability in the following theorem

\begin{theorem}
The drug-free equilibrium $E_f$ is globally asymptotically stable if $\mathcal{R}_0 < 1$. While the drug-addiction equilibrium $E^*$ is globally asymptotically stable if $\mathcal{R}_0 > 1$.
\end{theorem}


\section{Extended cocaine-heroin SCHR model}\label{S7}
The main objective of this section is to present an extension of the proposed cocaine-heroin SCHR model \eqref{E2.1}, cf.\ Figure~\ref{F1}, and to derive the main results associated with it. Based on the parameters mentioned in Table~\ref{Tab1}, Table~\ref{Tab2} lists the additional parameters for the new model, detailing their respective roles and significance.

\begin{table}[hbtp]
\centering
\setlength{\tabcolsep}{0.5cm}
\caption{Additional transmission coefficients for the extended cocaine-heroin SCHR model.}\label{Tab2}
\begin{tabular}{cc}
\hline 
\quad Symbol \quad & \quad Description \quad \\
\hline \hline 
$\eta_5$ and $\eta_6$ & Death rates of $U_c$ and $U_h$ respectively\\ 
\hline
$d$ & Diffusion rate of $S$, $C$, $U_c$, $H$, $U_h$, and $R$\\
\hline
$\gamma_3$ and $\gamma_4$ & Natural recovery rates of $U_c$ and $U_h$ respectively\\
\hline
$\mu_1$ & Progression rate to $U_c$ from $C$\\
\hline
$\kappa_1$ & Progression rate to $U_h$ from $H$\\
\hline
$\mu_2$ & Progression rate to $C$ from $U_c$\\
\hline
$\kappa_2$ & Progression rate to $H$ from $U_h$\\
\hline
\end{tabular}
\end{table}

Figure~\ref{F2} provides a summary of how the epidemic spreads from one compartment to another, where $U_c$ and $U_h$ are the number of cocaine and heroin users undergoing treatment, respectively.

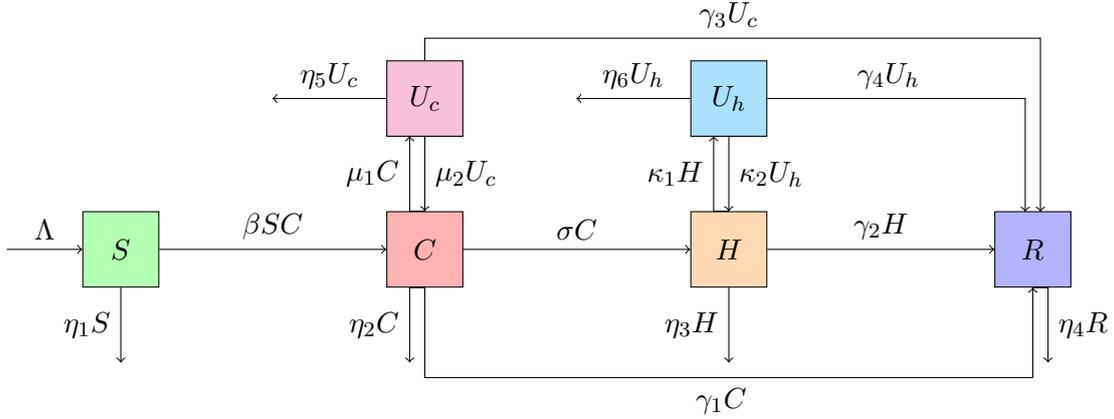
\begin{figure}[hbtp]
\centering
\begin{tikzpicture}[node distance=4cm]
\node (S) [rectangle, draw, minimum size=1cm, fill=green!30] {$S$}; 
\node (C) [rectangle, draw, minimum size=1cm, fill=red!30, right of=S] {$C$};
\node (UC) [rectangle, draw, minimum size=1cm, fill=magenta!30, above of = C, yshift=-2cm] {$U_c$};
\node (H) [rectangle, draw, minimum size=1cm, fill=orange!30, right of=C] {$H$};
\node (UH) [rectangle, draw, minimum size=1cm, fill=cyan!30, above of=H, yshift=-2cm] {$U_h$};
\node (R) [rectangle, draw, minimum size=1cm, fill=blue!30, right of=H] {$R$}; 
\draw[->] (-1.5,0) -- ++(S) node[midway,above]{$\Lambda$};
\draw [->] (S) -- (C) node[midway,above]{$\beta SC$};
\draw [->] (C) -- (H) node[midway,above]{$\sigma C$};
\draw [->] (H) -- (R) node[midway,above]{$\gamma_2 H$};
\draw [->] (C.south) -| (4,-1.7) -- (11.8,-1.7) node[midway,below]{$\gamma_1 C$} -| (R.south);
\draw[->] (C.north) -| (3.8,1.5) node[near end,left]{$\mu_1 C$};
\draw[->] (UC.south) -| (C.north) node[near end,right]{$\mu_2 U_c$};
\draw [->] (UC.north) -| (4,2.8) -- (12,2.8) node[above, xshift=-4cm]{$\gamma_3 U_c$} -| (12.1,0.5);
\draw[->] (H.north) -| (7.8,1.5) node[near end,left]{$\kappa_1 H$};
\draw[->] (UH.south) -| (H.north) node[near end,right]{$\kappa_2 U_h$};
\draw [->] (UH) -| (11.9,2) -- (11.9,2) node[above, xshift=-1.8cm]{$\gamma_4 U_h$} -| (11.9,0.5);
\draw[->] (S.south) -| (0,-1.5) node[near end,left]{$\eta_1 S$};
\draw[->] (C.south) -| (3.8,-1.5) node[near end,left]{$\eta_2 C$};
\draw[->] (H.south) -| (8,-1.5) node[near end,left]{$\eta_3 H$};
\draw[->] (R.south) -| (12.2,-1.5) node[near end,right]{$\eta_4 R$};
\draw[->] (UC.west) -- ++(-1.5,0) node[midway,above]{$\eta_5 U_c$};
\draw[->] (UH.west) -- ++(-1.5,0) node[midway,above]{$\eta_6 U_h$};
\end{tikzpicture}
\captionof{figure}{Transfer diagram for the extended cocaine-heroin SCHR model.}\label{F2}
\end{figure}

Given the assumptions outlined above, the extended reaction-diffusion cocaine-heroin SCHR model can be formulated as follows
\begin{equation}\label{E7.1}
\left\{\begin{aligned}
\partial_t S - d \Delta S &= \Lambda - \beta S C - \eta_1 S, \\
\partial_t C - d \Delta C &= \beta S C - (\eta_2+\sigma+\gamma_1) C + \mu_2 U_c - \mu_1 C,\\ 
\partial_t U_c - d \Delta U_c &= \mu_1 C - (\mu_2 + \eta_5 + \gamma_3) U_c,\\ 
\partial_t H - d \Delta H &= -(\eta_3+\gamma_2) H + \sigma C + \kappa_2 U_h - \kappa_1 H,\\ 
\partial_t U_h - d \Delta U_h &= \kappa_1 H - (\kappa_2 + \eta_6 + \gamma_4) U_h,\\ 
\partial_t R - d \Delta R &= \gamma_1 C + \gamma_2 H - \eta_4 R + \gamma_3 U_c + \gamma_4 U_h,
\end{aligned}\right. \ \text{ in } \Omega_b, 
\end{equation}
with the no-flux boundary conditions 
\begin{equation}\label{E7.2} 
  \frac{\partial S}{\partial \nu} = \frac{\partial C}{\partial \nu} 
   = \frac{\partial U_c}{\partial \nu} = \frac{\partial H}{\partial \nu} 
   = \frac{\partial U_h}{\partial \nu} = \frac{\partial R}{\partial \nu} = 0, \ \text{ on } \Sigma_b,
\end{equation}
and the initial conditions
\begin{equation}\label{E7.3} 
\begin{aligned}
   &S(0,x)=S^0, \ C(0,x)=C^0, \ U_c(0,x)=U_c^0,\\
   &H(0,x)=H^0, \ U_h(0,x)=U_h^0, \ R(0,x)=R^0,
\end{aligned} \quad \text{ in } \mathcal{U}.
\end{equation}

\subsection{Existence of the strong solution}

Let $\vartheta = (\vartheta_1, \vartheta_2, \vartheta_3, \vartheta_4, \vartheta_5, \vartheta_6) = (S, C, U_c, H, U_h, R)$, 
$\vartheta^0 =( \vartheta^0_1, \vartheta^0_2, \vartheta^0_3, \vartheta^0_4, \vartheta^0_5, \vartheta^0_6)$, 
and $\mathbb{W}(\mathcal{U}) = \bigl( L^2(\mathcal{U}) \bigr)^6$. 
The extended problem \eqref{E7.1}--\eqref{E7.3} can be reformulated in $\mathbb{W}( \mathcal{U})$ as follows
\begin{equation*}
\begin{cases}
   \partial_t \vartheta(t) + \mathcal{A}\vartheta(t) = \tilde{\xi}(\vartheta(t)),\\
    \vartheta(0)=\vartheta^0,
\end{cases} \ t\in[0, b],
\end{equation*}
where
\begin{equation*}
     \mathcal{A}\colon \ 
\begin{aligned}
     &\mathcal{D}_\mathcal{A} 
     = \Big\{\vartheta \in \bigl( H^2(\mathcal{U}) \bigr)^6 \mid \ \frac{\partial \vartheta_i}{\partial \nu} = 0, 1\le i\le 6\Big\} \subset \mathbb{W}(\mathcal{U}) \longrightarrow \mathbb{W}(\mathcal{U}),\\ 
&\vartheta \longrightarrow -d\Delta \vartheta = (-d\Delta \vartheta_i)_{i=1, \dots, 6},
\end{aligned}
\end{equation*}
and
\begin{equation*}
\tilde{\xi}(\vartheta(t))=
\begin{pmatrix}
\Lambda - \beta \vartheta_1\vartheta_2 - \eta_1 \vartheta_1\\
\beta \vartheta_1\vartheta_2 + \mu_2 \vartheta_3 - (\eta_2 + \sigma + \gamma_1 + \mu_1)\vartheta_2\\
\mu_1 \vartheta_2 - (\eta_5 + \gamma_3 + \mu_2)\vartheta_3\\
\sigma \vartheta_2 + \kappa_2 \vartheta_5 - (\eta_3 + \gamma_2 + \kappa_1)\vartheta_4\\
\kappa_1 \vartheta_4 - (\eta_6 + \gamma_4 + \kappa_2)\vartheta_5\\
\gamma_1\vartheta_2 + \gamma_2\vartheta_4 + \gamma_3\vartheta_3 + \gamma_4\vartheta_5 - \eta_4\vartheta_6
\end{pmatrix}, \quad t\in [0, b].
\end{equation*}
Using the same techniques described in Section~\ref{S3}, we can conclude the following theorem.

\begin{theorem}
The system \eqref{E7.1}--\eqref{E7.3} has a unique positive bounded strong solution $\vartheta\in W^{1,2} (0, T ; \mathbb{W}(\mathcal{U}))$. 
Furthermore, for each $i\in\{1, \dots, 6\}$ we get
\begin{equation*}
   \vartheta_i\in L^2\bigl(0, T ; H^2(\mathcal{U})\bigr) \cap 
   L^\infty\bigl(0, T ; H^1(\mathcal{U})\bigr) \cap L^\infty\bigl(\Omega_b\bigr).
\end{equation*}
\end{theorem}


\subsection{Stability analysis}
The drug-free equilibrium for the given system \eqref{E7.1} is obtained by setting the right-hand sides of the equations to zero and solving for the state variables when there are no drugs present in the system. 
This means that the concentrations of drug users ($C$, $U_c$, $H$, $U_h$) are zero. 
Thus, the drug-free equilibrium $E_f$ is
\begin{equation*}
    E_f = \Bigl( \frac{\Lambda}{\eta_1}, 0, 0, 0, 0, 0 \Bigr).
\end{equation*}
Following the same steps as in Section~\ref{S4}, using the transition matrix $\tilde{\mathcal{K}}$ and the transmission matrix $\tilde{\mathcal{T}}$ at point $E_f$, the basic reproduction number for the system \eqref{E7.1} is given by
\begin{equation*}
    \mathcal{R}_0 
    = \operatorname{trace}(-\tilde{\mathcal{T}}\tilde{\mathcal{K}}^{-1}) 
    = \frac{\beta \Lambda}{\eta_1 (\eta_2 + \sigma + \gamma_1 + \mu_1)}.
\end{equation*}
To find the drug-addiction equilibrium of the given system, we must set the time derivatives to zero and solve the resulting algebraic equations for the equilibrium values $S^*$, $C^*$, $U_c^*$, $H^*$, $U_h^*$, $R^*$. 
Therefore, the drug-addiction equilibrium $E^*$ is
\begin{equation*}
   E^* = (S^*, C^*, U_c^*, H^*, U_h^*, R^*),
\end{equation*}
where
\begin{align*}
S^* &= \frac{\Lambda}{\eta_1 \mathcal{R}_0} - \frac{\mu_1 \mu_2}{\beta(\mu_2 + \eta_5 + \gamma_3)}, \\
C^* &= \frac{\Lambda - \eta_1 S^*}{\beta S^*},\\
U_c^* &= \frac{\mu_1 C^*}{\mu_2 + \eta_5 + \gamma_3},\\
H^* &= \frac{\sigma(\kappa_2 + \eta_6 + \gamma_4) C^*}{(\kappa_1 + \eta_3 + \gamma_2)(\kappa_2 + \eta_6 + \gamma_4) - \kappa_1\kappa_2},\\
U_h^* &= \frac{\kappa_1 H^*}{\kappa_2 + \eta_6 + \gamma_4},\\
R^* &= \frac{\gamma_1 C^* + \gamma_2 H^* + \gamma_3 U_c^* + \gamma_4 U_h^*}{\eta_4}.
\end{align*}
An important observation is that the variable $R$ does not appear in the first five equations of \eqref{E7.1}. 
This allows us to focus our analysis on the following subsystem
\begin{equation}\label{E7.4}
\left\{\begin{aligned}
&\begin{aligned}
&\partial_t S - d \Delta S = \Lambda - \beta S C - \eta_1 S, \\
&\partial_t C - d \Delta C = \beta S C - (\eta_2+\sigma+\gamma_1) C + \mu_2 U_c - \mu_1 C,\\ 
&\partial_t U_c - d \Delta U_c = \mu_1 C - (\mu_2 + \eta_5 + \gamma_3) U_c,\\ 
&\partial_t H - d \Delta H = -(\eta_3+\gamma_2) H + \sigma C + \kappa_2 U_h - \kappa_1 H,\\ 
&\partial_t U_h - d \Delta U_h = \kappa_1 H - (\kappa_2 + \eta_6 + \gamma_4) U_h,\\
\end{aligned} \quad \text{ in } \Omega_b,\\
&\frac{\partial S}{\partial \nu} = \frac{\partial C}{\partial \nu} = \frac{\partial U_c}{\partial \nu} = \frac{\partial H}{\partial \nu} = \frac{\partial U_h}{\partial \nu} = 0, \quad \text{ on } \Sigma_b,\\ \vspace{0.15cm}
&\bigl(S(0,x), C(0,x), U_c(0,x), H(0,x), U_h(0,x)\bigr) 
 = \bigl(S^0, C^0, U_c^0, H^0, U_h^0\bigr), \quad \text{ in } \mathcal{U}.
\end{aligned}\right.
\end{equation}


\subsubsection{Global stability of the drug-free equilibrium}
Consider the Lyapunov function 
\begin{equation*}
   \tilde{\mathcal{F}}(\vartheta) = C + \alpha_1 U_c + \alpha_2 H + \alpha_3 U_h ,
\end{equation*} 
where $ \alpha_1, \alpha_2, \alpha_3 $ are positive constants to be determined. By computing the time derivative of $\tilde{\mathcal{F}}$ along the trajectories of the system, we get
\begin{equation*}
\begin{split}
\frac{d \tilde{\mathcal{F}}}{d t}
&= \nabla \tilde{\mathcal{F}}(\vartheta) \cdot \tilde{\xi}(\vartheta)\\
&= \beta S C - (\eta_2 + \sigma + \gamma_1) C + \mu_2 U_c - \mu_1 C 
   + \alpha_1 \bigl( \mu_1 C - (\mu_2 + \eta_5 + \gamma_3) U_c \bigr) \\
&\quad + \alpha_2 \bigl( -(\eta_3 + \gamma_2) H + \sigma C + \kappa_2 U_h - \kappa_1 H \bigr) 
+ \alpha_3 \bigl( \kappa_1 H - (\kappa_2 + \eta_6 + \gamma_4) U_h \bigr).
\end{split}
\end{equation*} 
Afterwards,
\begin{equation*}
\begin{split}
\frac{d \tilde{\mathcal{F}}}{d t} 
&= \bigl(\beta S - (\eta_2 + \sigma + \gamma_1 + \mu_1)\bigr) C 
   + \bigl(\mu_2 - \alpha_1 (\mu_2 + \eta_5 + \gamma_3)\bigr) U_c \\
&\quad + \bigl(\sigma - \alpha_2 (\eta_3 + \gamma_2) + \alpha_2 \kappa_2\bigr) C
  + \bigl(\alpha_2 \kappa_2 - \alpha_3 (\kappa_2 + \eta_6 + \gamma_4)\bigr) U_h.
\end{split}
\end{equation*} 
Then, a Lyapunov functional associated with \eqref{E7.4} at the drug-free equilibrium $E_f$ is defined by
\begin{equation*}
   \tilde{\mathcal{G}} = \int_{\mathcal{U}} \tilde{\mathcal{F}}(\vartheta(t, x)) \,dx.
\end{equation*} 
Then, using similar techniques outlined in \cite{Hattaf2013}, we have
\begin{equation*}
\begin{split}
\frac{d \tilde{\mathcal{G}}}{d t}
&= \int_{\mathcal{U}} \Big[\bigl(\beta S - (\eta_2 + \sigma + \gamma_1 + \mu_1)\bigr) C 
+ \bigl(\mu_2 - \alpha_1 (\mu_2 + \eta_5 + \gamma_3)\bigr) U_c \\
&\quad +  \bigl(\sigma - \alpha_2 (\eta_3 + \gamma_2) + \alpha_2 \kappa_2\bigr) C
+ \bigl(\alpha_2 \kappa_2 - \alpha_3 (\kappa_2 + \eta_6 + \gamma_4)\bigr) U_h \Big] \,dx\\
&\quad - d S^f \int_{\mathcal{U}} \frac{|\nabla S|^2}{S^2}\,dx,
\end{split}
\end{equation*} 
where $S^f = \frac{\Lambda}{\eta_1}$. We can now choose $\alpha_1$, $\alpha_2$, $\alpha_3 $ so that all coefficients of $C$, $U_c$, $H$ and $U_h$ are negative.
Thus, $\frac{d \tilde{\mathcal{G}}}{d t} < 0$.
By LaSalle's Invariance Principle, $E_f$ is globally asymptotically stable if $\mathcal{R}_0 < 1$.


\subsubsection{Global stability of the drug-addiction equilibrium}
To show that the drug-addiction equilibrium $E^*$ is globally asymptotically stable when $\mathcal{R}_0 > 1$, we use the Lyapunov function
\begin{equation*}
  \hat{\mathcal{F}}(\vartheta) 
  = (S - S^*)^2 + (C - C^*)^2 + (U_c - U_c^*)^2 + (H - H^*)^2 + (U_h - U_h^*)^2.
\end{equation*} 
Then,
\begin{equation*}
\frac{d \hat{\mathcal{F}}}{dt} 
= 2(S - S^*) \frac{\partial S}{\partial t} + 2(C - C^*) \frac{\partial C}{\partial t} + 2(U_c - U_c^*) \frac{\partial U_c}{\partial t} + 2(H - H^*) \frac{\partial H}{\partial t} + 2(U_h - U_h^*) \frac{\partial U_h}{\partial t}.
\end{equation*} 
The associated ODE system of \eqref{E7.4}, yields
\begin{equation*}
\begin{split}
\frac{d \hat{\mathcal{F}}}{dt} 
&= 2(S - S^*) ( \Lambda - \beta S C - \eta_1 S) + 2(C - C^*)
   \bigl( \beta S C - (\eta_2 + \sigma + \gamma_1 + \mu_1) C + \mu_2 U_c \bigr) \\
&\quad + 2(U_c - U_c^*) \bigl( \mu_1 C - (\mu_2 + \eta_5 + \gamma_3) U_c \bigr)
+ 2(H - H^*) \bigl( \sigma C + \kappa_2 U_h - (\eta_3 + \gamma_2 + \kappa_1) H \bigr) \\
& \quad + 2(U_h - U_h^*) \left( \kappa_1 H - (\kappa_2 + \eta_6 + \gamma_4) U_h \right).
\end{split}
\end{equation*} 
Since
\begin{align*}
\Lambda &= \beta S^* C^* + \eta_1 S^*, \\
\beta S^* C^* &= (\eta_2 + \sigma + \gamma_1 + \mu_1) C^* - \mu_2 U_c^*, \\
\mu_1 C^* &= (\mu_2 + \eta_5 + \gamma_3) U_c^*, \\
\sigma C^* + \kappa_2 U_h^* &= (\eta_3 + \gamma_2 + \kappa_1) H^*, \\
\kappa_1 H^* &= (\kappa_2 + \eta_6 + \gamma_4) U_h^*,
\end{align*}
we obtain
\begin{equation*}
\begin{split}
\frac{d \hat{\mathcal{F}}}{dt} 
&= - 2(S - S^*) \bigl( \beta (S C - S^* C^*) + \eta_1 (S - S^*) \bigr)\\
&\quad + 2(C - C^*) \bigl( \beta (S C - S^* C^*) + \mu_2 (U_c - U_c^*) - (\eta_2 + \sigma + \gamma_1 + \mu_1) (C - C^*) \bigr) \\
&\quad + 2(U_c - U_c^*) \bigl( \mu_1 (C - C^*) - (\mu_2 + \eta_5 + \gamma_3) (U_c - U_c^*) \bigr)\\
& \quad + 2(H - H^*) \bigl( \sigma (C - C^*) + \kappa_2 (U_h - U_h^*) - (\eta_3 + \gamma_2 + \kappa_1) (H - H^*) \bigr) \\
&\quad + 2(U_h - U_h^*) \bigl( \kappa_1 (H - H^*) - (\kappa_2 + \eta_6 + \gamma_4) (U_h - U_h^*) \bigr).
\end{split}
\end{equation*} 
So,
\begin{equation*}
    \hat{\mathcal{G}} = \int_{\mathcal{U}} \hat{\mathcal{F}}\bigl(\vartheta(t, x)\bigr) \,dx.
\end{equation*} 
is a Lyapunov functional associated with \eqref{E7.4} at the drug-addiction equilibrium $E^*$.
Then we have,
\begin{equation*}
\begin{split}
\frac{d \hat{\mathcal{G}}}{d t} 
&= \int_{\mathcal{U}} \Big[ -2(S - S^*) \bigl( \beta (S C - S^* C^*) + \eta_1 (S - S^*) \bigr)\\
&\qquad + 2(C - C^*) \bigl( \beta (S C - S^* C^*) + \mu_2 (U_c - U_c^*) - (\eta_2 + \sigma + \gamma_1 + \mu_1) (C - C^*) \bigr) \\
&\qquad + 2(U_c - U_c^*) \bigl( \mu_1 (C - C^*) - (\mu_2 + \eta_5 + \gamma_3) (U_c - U_c^*) \bigr)\\
&\qquad + 2(H - H^*) \bigl( \sigma (C - C^*) + \kappa_2 (U_h - U_h^*) - (\eta_3 + \gamma_2 + \kappa_1) (H - H^*) \bigr) \\
&\qquad + 2(U_h - U_h^*) \bigl( \kappa_1 (H - H^*) - (\kappa_2 + \eta_6 + \gamma_4) (U_h - U_h^*) \bigr) \Big] \,dx\\
&\quad - d S^* \int_{\mathcal{U}} \frac{|\nabla S|^2}{S^2}\,dx 
       - d C^* \int_{\mathcal{U}} \frac{|\nabla C|^2}{C^2}\,dx 
       - d U_c^* \int_{\mathcal{U}} \frac{|\nabla U_c|^2}{U_c^2}\,dx\\
&\quad - d H^* \int_{\mathcal{U}} \frac{|\nabla H|^2}{H^2}\,dx 
       - d U_h^* \int_{\mathcal{U}} \frac{|\nabla U_h|^2}{U_h^2}\,dx.
\end{split}
\end{equation*} 
We can see that each term involving deviations from the equilibrium values is either negative or zero, which guarantees that $\frac{d \hat{\mathcal{G}}}{d t} \leq 0$. Therefore, $E^*$ is globally asymptotically stable if $\mathcal{R}_0 > 1$.


\section{Numerical results}\label{S8}
The main objective of this section is to present approximate and numerical results derived from the proposed epidemiological models \eqref{E2.1}--\eqref{E2.2} and \eqref{E7.1}--\eqref{E7.3}. 
The provided mathematical model \eqref{E2.1} (or the extended one \eqref{E7.1}) was solved numerically using a finite difference method implemented in MatLab with a single spatial variable $x$ and time $t$. 
On the first day of drug use, we assume that susceptible individuals are evenly distributed over the studied area. 
Our numerical simulations used some specific parameter values and initial data from \cite{din2020controlling, Duan2020, wang2011dynamics, zhang2020stability, Zinihi2024MM}, as detailed in Tables~\ref{Tab3} and \ref{Tab4}.

\begin{table}[hbtp]
\centering
\setlength{\tabcolsep}{0.45cm}
\caption{Initial conditions and parameter values of the proposed cocaine-heroin model \eqref{E2.1}--\eqref{E2.2}.}\label{Tab3}
\begin{tabular}{ccc}
\hline Symbol & Description & Value\\
\hline\hline $S^0$ & Initial susceptible individuals & $30$\\ 
\hline $C^0$ & Initial cocaine users & $10$\\ 
\hline $H^0$ & Initial heroin users & $5$\\
\hline $R^0$ & Initial recovered individuals & $0$\\
\hline $\eta_1 = \eta_2 = \eta_3 = \eta_4$ & Natural death rates & $0.01$ (for $\mathcal{R}_0 > 1$)\\ & & $0.03$ (for $\mathcal{R}_0 < 1$)\\
\hline $\sigma$ & Recovery rate of cocaine to heroin users & $0.2$ (for $\mathcal{R}_0 > 1$)\\ & & $0.001$ (for $\mathcal{R}_0 < 1$)\\
\hline $\beta$ & Transmission rate & $0.002$ (for $\mathcal{R}_0 > 1$)\\ & & $0.001$ (for $\mathcal{R}_0 < 1$)\\
\hline $d$ & Diffusion coefficient & $0.1$\\ 
\hline $T$ & Final time & $500$\\
\hline $\gamma_1 = \gamma_2$ & Natural recovery rates of $C$ and $H$ users & $0.05$\\
\hline $\Lambda$ & Recruitment rate of the population & $2.15$\\
\hline
\end{tabular}
\end{table}

In addition to the parameter values and initial data shown in Table~\ref{Tab3}, Table~\ref{Tab4} contains other parameters and initial conditions used to determine the numerical results for our extended model.

\begin{table}[hbtp]
\centering
\setlength{\tabcolsep}{0.45cm}
\caption{Initial conditions and parameter values of the extended cocaine-heroin model \eqref{E7.1}--\eqref{E7.3}.}\label{Tab4}
\begin{tabular}{ccc}
\hline Symbol & Description & Value\\ 
\hline $U_c^0$ & Initial cocaine users undergoing treatment & $3$\\ 
\hline $U_h^0$ & Initial heroin users undergoing treatment & $3$\\
\hline $\eta_5 = \eta_6$ & Natural death rates & $0.01$ \\
\hline $\gamma_3 = \gamma_4$ & Natural recovery rates of $U_c$ and $U_h$ & $0.03$\\
\hline
$\mu_1 = \mu_2$ & Progression rates of $C$ and $U_c$ & $0.01$ (for $\mathcal{R}_0 > 1$)\\ & & $0.05$ (for $\mathcal{R}_0 < 1$)\\
\hline
$\kappa_1 = \kappa_2$ & Progression rates of $H$ and $U_h$ & $0.01$\\
\hline
\end{tabular}
\end{table}

The spatial variable $0\le x\le2$ is discretized into a grid of $M_x + 1$ equidistant points $x_j = k\delta_x$ 
for $j = 0, 1, \dots, M_x$, where $\delta_x = \frac{2}{M_x}$ represents the uniform spatial grid spacing. 
Additionally, the diffusion term has been discretized using the standard second-order difference quotient
\begin{equation*}
    \Delta \phi \approx \frac{\phi_{j+1} - 2\phi_j + \phi_{j-1}}{\delta_x^2},
\end{equation*}
with $d = 0.1$. 
The solution of the proposed system is computed until reaching a steady state (i.e., $T = 500$). 
The grid spacing is defined as $\delta_t = 10^{-2}$.


\subsection{Drug-free equilibrium}
Figures~\ref{F3} and \ref{F4} show the time series of the model solution for the scenario where $\mathcal{R}_0 \leq 1$, indicating the global asymptotic stability of the drug-free equilibrium $E_f$ for both \eqref{E2.1}--\eqref{E2.2} and \eqref{E7.1}--\eqref{E7.3}, respectively.

\begin{figure}[hbtp]
\centering
\includegraphics[scale=0.35]{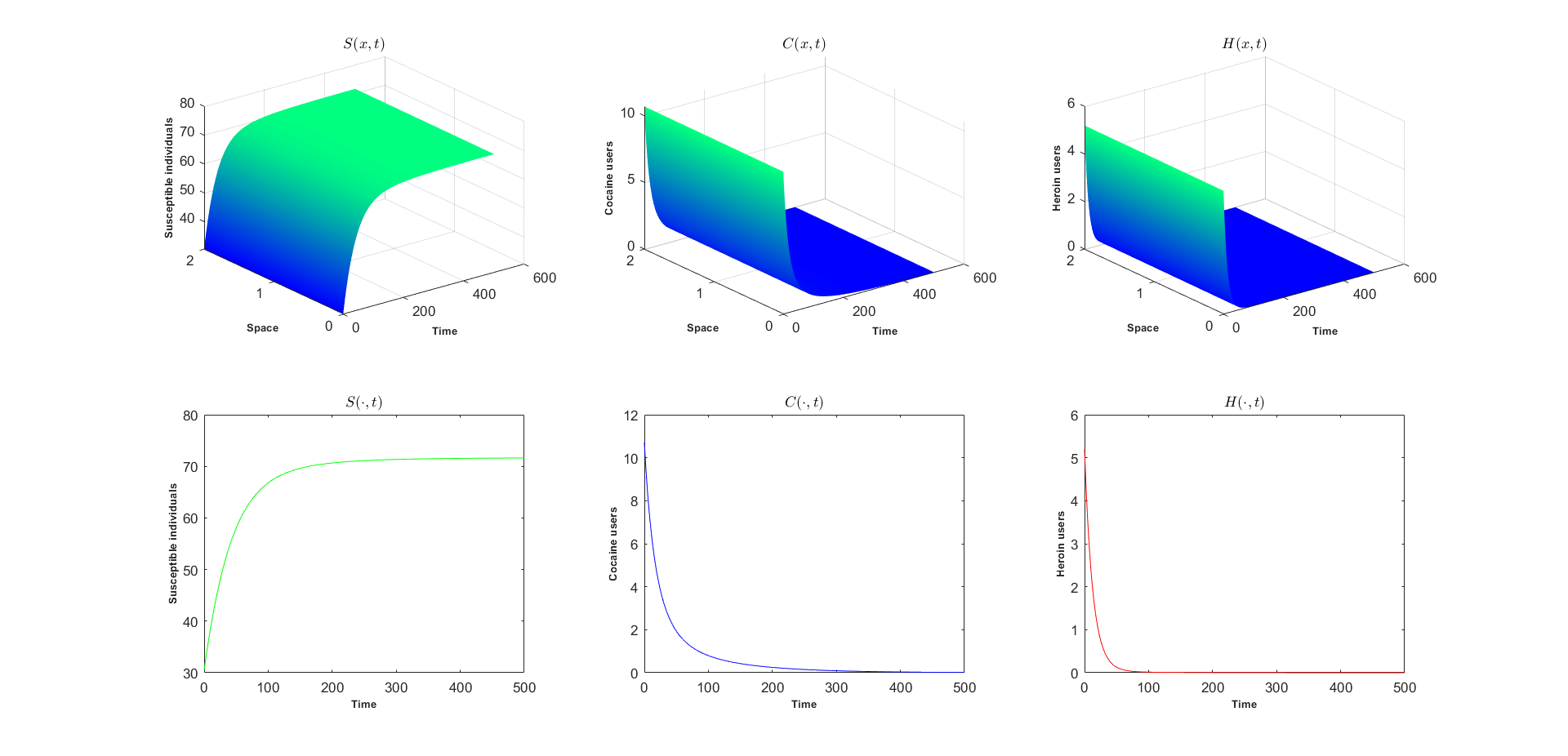}
\caption{Numerical results of \eqref{E2.1}--\eqref{E2.2} when $\eta_1 = \eta_2 = 3\times 10^{-2}$ and $\sigma = \beta = 10^{-3}$ ($\mathcal{R}_0 < 1$).}\label{F3}
\end{figure}

\begin{figure}[hbtp]
\centering
\includegraphics[scale=0.35]{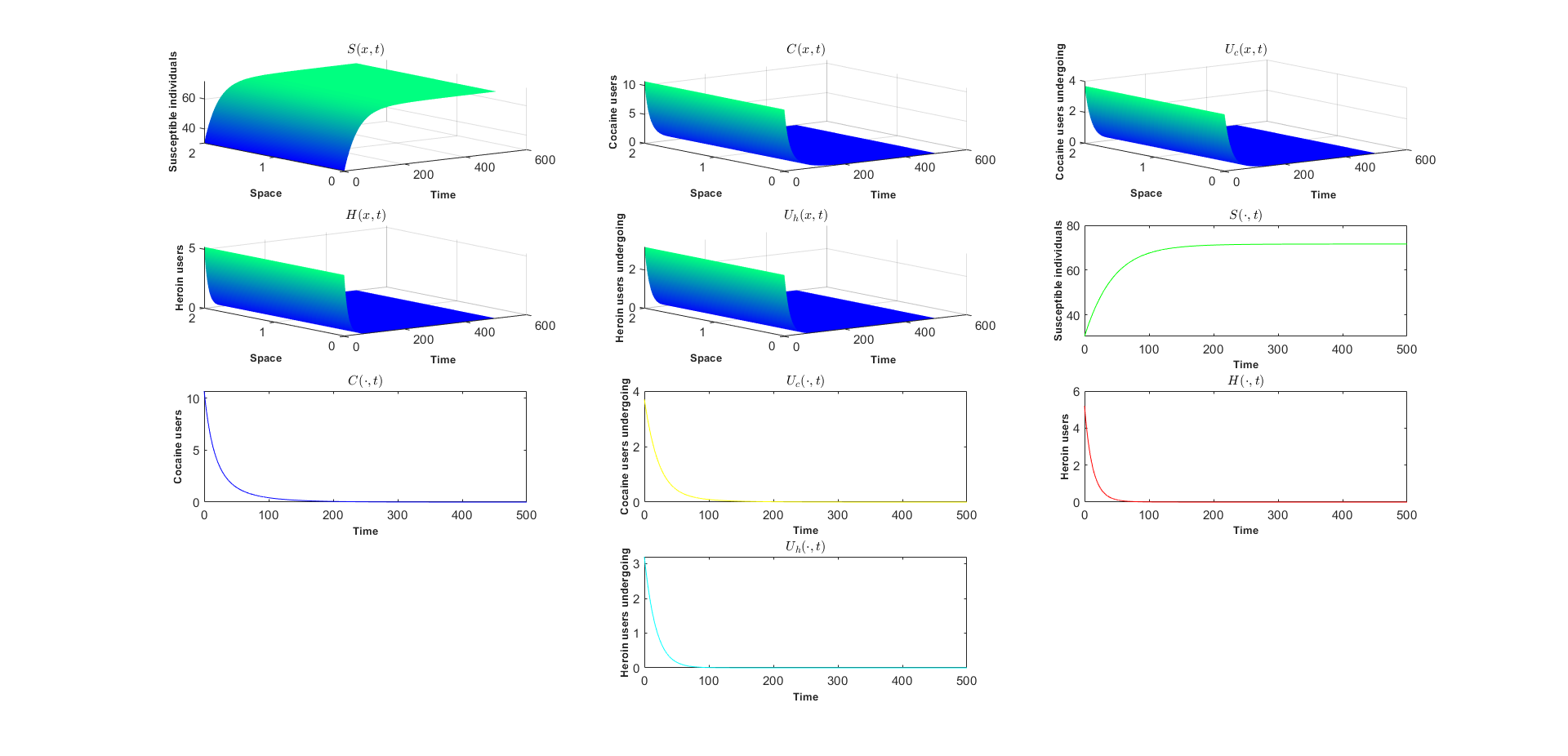}
\caption{Numerical results of \eqref{E7.1}--\eqref{E7.3} when $\eta_1 = \eta_2 = 3\times 10^{-2}$, $\mu_1 = 0.05$ and $\sigma = \beta = 10^{-3}$ ($\mathcal{R}_0 < 1$).}\label{F4}
\end{figure}


\subsection{Drug-addiction equilibrium}
Figures~\ref{F5} and \ref{F6} show the time series of the solution for the case where $\mathcal{R}_0 \ge 1$, illustrating the global asymptotic stability of the drug-addiction equilibrium $E^*$ for both systems \eqref{E2.1}--\eqref{E2.2} and \eqref{E7.1}--\eqref{E7.3}.

\begin{figure}[hbtp]
\centering
\includegraphics[scale=0.35]{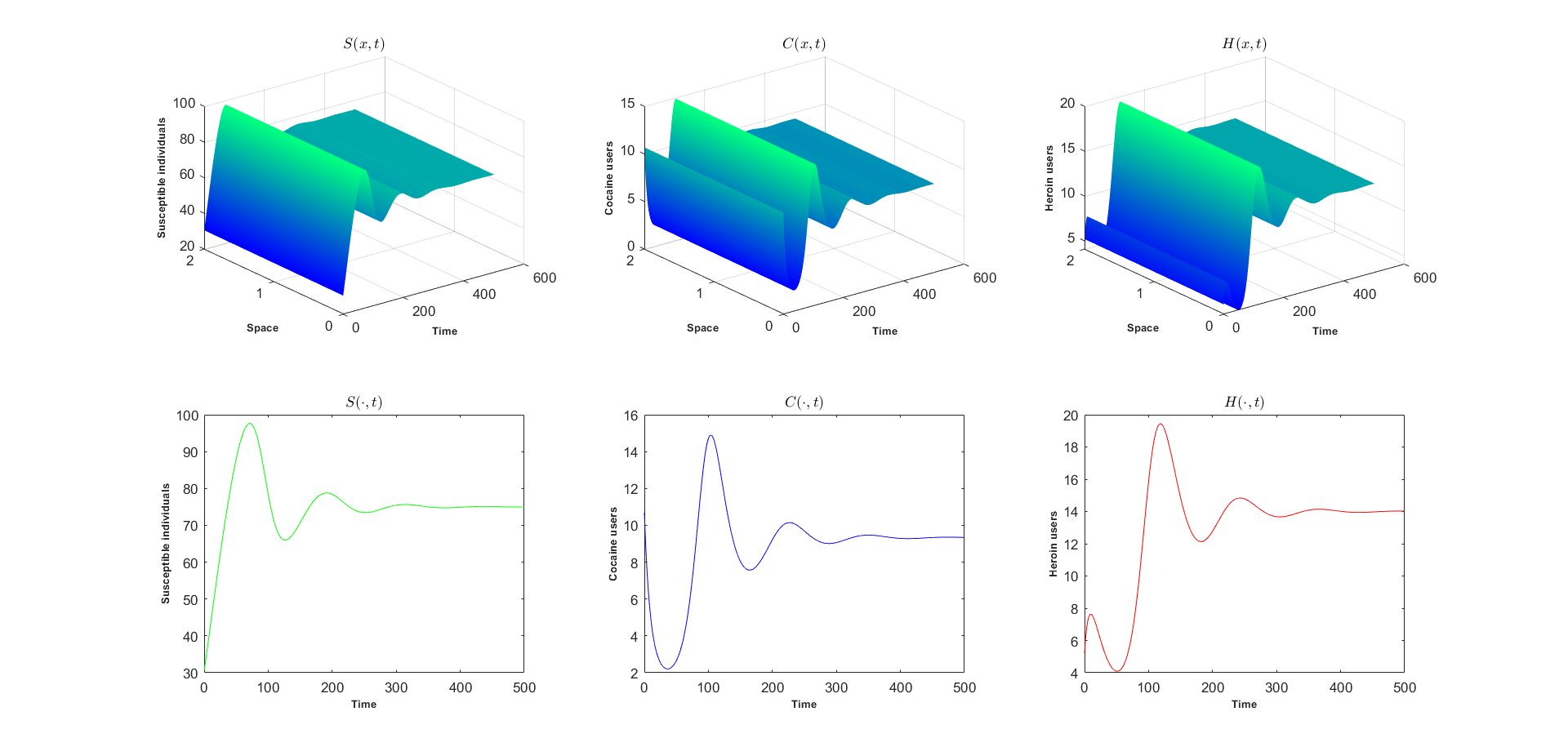}
\caption{Numerical results of \eqref{E2.1}--\eqref{E2.2} when $\eta_1 = \eta_2 = 10^{-2}$, $\sigma = 0.2$, and $\beta = 2 \times 10^{-3}$ ($\mathcal{R}_0 > 1$).}\label{F5}
\end{figure}

\begin{figure}[hbtp]
\centering
\includegraphics[scale=0.35]{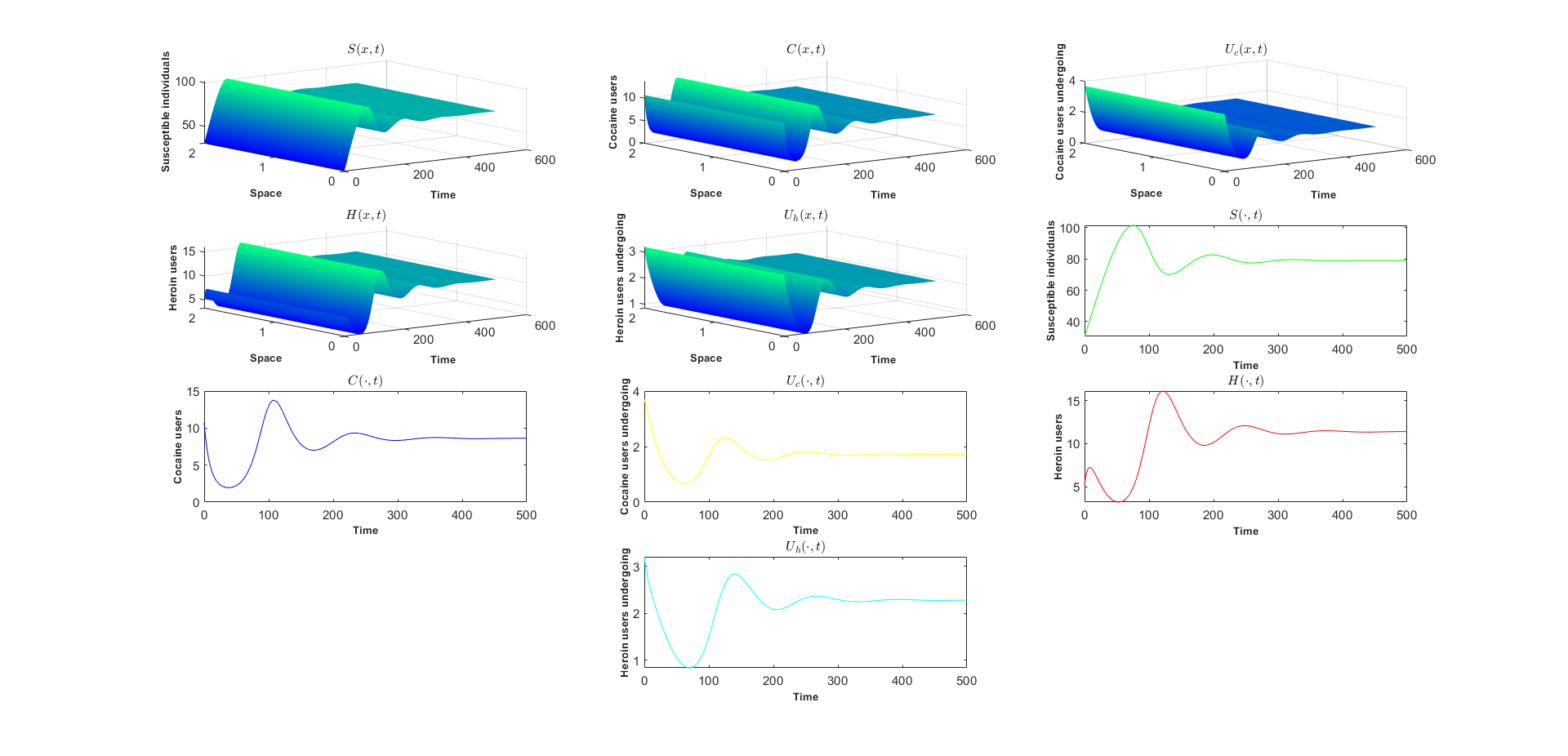}
\caption{Numerical results of \eqref{E7.1}--\eqref{E7.3} when $\eta_1 = \eta_2 = 10^{-2}$, $\sigma = 0.2$, $\mu_1 = 0.01$ and $\beta = 2 \times 10^{-3}$ ($\mathcal{R}_0 > 1$).}\label{F6}
\end{figure}


\section{Conclusion}\label{S9}
This study contributes to the ongoing advances in understanding the spatial diffusion dynamics of epidemiological models. 
Our approach involves the use of a set of PDEs to model the complex interactions within the four compartments. 
To ensure the robustness of our results, we systematically establish the existence, uniqueness, 
and boundedness of a strong non-negative solution to the given problem.
We also provide a comprehensive analysis of the local and global stability of the equilibria. 
Moreover, we introduce and analyze an extended cocaine-heroin model related to the proposed one. 
Additionally, we present a practical application of our methodology through a numerical study focusing on the global stability assessment of a reaction-diffusion SCHR model.

Future research directions will be twofold. First, we will study further extensions of our model, e.g.\ nonlinear incidence functions, time-distributed delay, 
time- or space-fractional models \cite{Zinihi2024MM, Zinihi2024OC}, 
age-multi-group model for heterogeneous population including treatment age, 
stochastic model \cite{Zinihi2024DB}, 
or optimal control \cite{SidiAmmi2023, Zinihi2024MM, Zinihi2024OC}. 
Secondly, we will design a nonstandard scheme
\cite[Section~4]{hoang2024differential} for the numerical solution of our proposed models which guarantees the positivity of the solution, the correct asymptotic behavior and (discrete) equilibrium points. In this respect, our goal will be to calibrate our model with real-world data to obtain realistic solutions.


\section*{Declarations}

\subsection*{Acknowledgments}
AZ, MRSA and AB would like to thank the Deanship of Research and Graduate Studies at King Khalid University
for funding this work through the Small group Research Project under grant number RGP1/120/45.\\
This work is carried out under the supervision of CNRST as part of the PASS program.


\subsection*{CRediT author statement} 

\textit{A. Zinihi:} Conceptualization, 
Methodology, 
Software, 
Formal analysis, 
Investigation, 
Writing -- Original Draft, 
Writing -- Review \& Editing, 
Visualization. 

\textit{M. R. Sidi Ammi:} Conceptualization, 
Methodology, 
Validation, 
Formal analysis, 
Investigation,  
Writing -- Original Draft, 
Writing -- Review \& Editing,  
Supervision.

\textit{M. Ehrhardt:} Validation, 
Formal analysis, 
Investigation, 
Writing -- Original Draft, 
Writing -- Review \& Editing, 
Project administration.

\textit{A. Bachir:} Validation, 
Formal analysis, 
Investigation, 
Writing \& Editing.


\subsection*{Data availability} 
All information analyzed or generated, which would support the results of this work are available in this article.
No data was used for the research described in the article.

\subsection*{Conflict of interest} 
The authors declare that there are no problems or conflicts 
of interest between them that may affect the study in this paper.


\bibliographystyle{acm}
\bibliography{paper}

\end{document}